\documentclass[a4paper]{article}

\usepackage{amsfonts}
\usepackage{amsmath}
\usepackage{amssymb}
\usepackage{amscd}
\usepackage{amsbsy}
\usepackage{amsthm}
\usepackage{bm}
\usepackage{empheq}
\usepackage{graphicx}
\usepackage{psfrag}
\usepackage{color}
\usepackage{cite}
\usepackage{multicol}
\usepackage{subfigure}
\usepackage[normalem]{ulem}
\usepackage{paralist}

\allowdisplaybreaks[4]

\numberwithin{equation}{section}
\catcode`@=11

\newtheoremstyle{thmstyle}% name
  {6pt}%      Space above
  {6pt}%      Space below
  {\it}%         Body font
  {}%         Indent amount (empty = no indent, \parindent = para indent)
  {\bf}% Thm head font
  {}%        Punctuation after thm head
  {.5em}%     Space after thm head: " " = normal interword space;
  {}%         Thm head spec (can be left empty, meaning `normal')
\newtheoremstyle{remstyle}% name
  {6pt}%      Space above
  {6pt}%      Space below
  {\rm}%         Body font
  {}%         Indent amount (empty = no indent, \parindent = para indent)
  {\bf}% Thm head font
  {}%        Punctuation after thm head
  {.5em}%     Space after thm head: " " = normal interword space;
  {}%         Thm head spec (can be left empty, meaning `normal')

\def\Section#1{\Sec{\large #1} \setcounter{equation}{0} \vskip -6mm \indent}
\def\Sec{\@Startsection{section}{1}{\z@}
                                   {-3.5ex \@plus -1ex \@minus -.2ex}%
                                   {2.3ex \@plus.2ex}%
                                   {\normalfont\large\bfseries\boldmath}}
\def\@Startsection#1#2#3#4#5#6{%
  \if@noskipsec \leavevmode \fi
  \par
  \@tempskipa #4\relax
  \@afterindenttrue
  \ifdim \@tempskipa <\z@
    \@tempskipa -\@tempskipa \@afterindentfalse
  \fi
  \if@nobreak
    \everypar{}%
  \else
    \addpenalty\@secpenalty\addvspace\@tempskipa
  \fi
  \@ifstar
    {\@ssect{#3}{#4}{#5}{#6}}%
    {\@dblarg{\@Sect{#1}{#2}{#3}{#4}{#5}{#6}}}}
\def\@Sect#1#2#3#4#5#6[#7]#8{%
  \ifnum #2>\c@secnumdepth
    \let\@svsec\@empty
  \else
    \refstepcounter{#1}%
    \protected@edef\@svsec{\@seccntformat{#1}\relax}%
  \fi
  \@tempskipa #5\relax
  \ifdim \@tempskipa>\z@
    \begingroup
      #6{%
          \@hangfrom{\hskip #3\relax\@svsec \hskip -2.5mm}%
          \interlinepenalty \@M #8\@@par}
    \endgroup
    \csname #1mark\endcsname{#7}%
    \addcontentsline{toc}{#1}{%
      \ifnum #2>\c@secnumdepth \else
        \protect\numberline{\csname the#1\endcsname}%
      \fi
      #7}%
  \else
    \def\@svsechd{%
      #6{\hskip #3\relax
      \@svsec #8}%
      \csname #1mark\endcsname{#7}%
      \addcontentsline{toc}{#1}{%
        \ifnum #2>\c@secnumdepth \else
          \protect\numberline{\csname the#1\endcsname}%
        \fi
        #7}}%
  \fi
  \@xsect{#5}}

\def\Subsection#1{\Subsec{#1} \vskip -6mm \indent}
\def\Subsec{\@StartSubsection{subsection}{2}{\z@}%
                                     {-3.25ex\@plus -1ex \@minus -.2ex}%
                                     {1.5ex \@plus .2ex}%
                                     {\normalfont\normalsize\bfseries\boldmath}}
\def\@StartSubsection#1#2#3#4#5#6{%
  \if@noskipsec \leavevmode \fi
  \par
  \@tempskipa #4\relax
  \@afterindenttrue
  \ifdim \@tempskipa <\z@
    \@tempskipa -\@tempskipa \@afterindentfalse
  \fi
  \if@nobreak
    \everypar{}%
  \else
    \addpenalty\@secpenalty\addvspace\@tempskipa
  \fi
  \@ifstar
    {\@ssect{#3}{#4}{#5}{#6}}%
    {\@dblarg{\@SubSect{#1}{#2}{#3}{#4}{#5}{#6}}}}
\def\@SubSect#1#2#3#4#5#6[#7]#8{%
  \ifnum #2>\c@secnumdepth
    \let\@svsec\@empty
  \else
    \refstepcounter{#1}%
    \protected@edef\@svsec{\@seccntformat{#1}\relax}%
  \fi
  \@tempskipa #5\relax
  \ifdim \@tempskipa>\z@
    \begingroup
      #6{%
          \@hangfrom{\hskip #3\relax\@svsec\hskip -1.5mm}%
          \interlinepenalty \@M #8\@@par}
    \endgroup
    \csname #1mark\endcsname{#7}%
    \addcontentsline{toc}{#1}{%
      \ifnum #2>\c@secnumdepth \else
        \protect\numberline{\csname the#1\endcsname}%
      \fi
      #7}%
  \else
    \def\@svsechd{%
      #6{\hskip #3\relax
      \@svsec #8}%
      \csname #1mark\endcsname{#7}%
      \addcontentsline{toc}{#1}{%
        \ifnum #2>\c@secnumdepth \else
          \protect\numberline{\csname the#1\endcsname}%
        \fi
        #7}}%
  \fi
  \@xsect{#5}}

\theoremstyle{thmstyle}
\newtheorem{thm}{\indent Theorem}[section]
\newtheorem{lem}[thm]{\indent Lemma}

\newtheorem{defi}[thm]{\indent Definition}

\newtheorem{prob}[thm]{\indent Problem}
\theoremstyle{remstyle}
\newtheorem{rem}[thm]{\indent Remark}
\newtheorem{algo}[thm]{\indent Algorithm}
\newtheorem{ex}[thm]{\indent Example}

\def\thebibliography#1{\section*{References}
\list{[\arabic{enumi}]} {\settowidth \labelwidth{[#1]} \leftmargin
\labelwidth \advance \leftmargin \labelsep \usecounter{enumi}}
\def\newblock{\hskip .11em plus .33em minus .07em} \small \sloppy \clubpenalty
4000 \widowpenalty 4000 \sfcode`\.=1000 \relax}

\def\BR{\mathbb R}

\def\rd{\mathrm d}

\def\T{\mathrm T}
\def\diam{\mathrm{diam}}

\def\supp{\mathrm{supp}}
\def\true{\mathrm{true}}
\def\Ga{\Gamma}

\def\Om{\Omega}
\def\al{\alpha}
\def\be{\beta}
\def\ga{\gamma}
\def\de{\delta}
\def\ep{\epsilon}
\def\ve{\varepsilon}
\def\te{\theta}
\def\ze{\zeta}
\def\ka{\kappa}

\def\om{\omega}
\def\f{\frac}
\def\nb{\nabla}

\def\pa{\partial}
\def\wh{\widehat}
\def\wt{\widetilde}
\def\tri{\triangle}

\title{{\Large\bf Numerical schemes for reconstructing profiles of moving sources in (time-fractional) evolution equations}}
\author{{\bf Yikan Liu}\\[2mm]
{\normalsize Research Institute for Electronic Science, Hokkaido University}}
\date{}
\begin{document}

\maketitle

\begin{abstract}
This article is concerned with the derivation of numerical reconstruction schemes for the inverse moving source problem on determining source profiles in (time-fractional) evolution equations. As a continuation of the theoretical result on the uniqueness, we adopt a minimization procedure with regularization to construct iterative thresholding schemes for the reduced backward problems on recovering one or two unknown initial value(s). Moreover, an elliptic approach is proposed to solve a convection equation in the case of two profiles.

\vskip 4.5mm

\noindent\begin{tabular}{@{}l@{ }p{10cm}}
{\bf Keywords } & Inverse moving source problem, (Time-fractional) evolution equation, Iterative thresholding scheme, Convection equation
\end{tabular}

\end{abstract}

\baselineskip 14pt
\setlength{\parindent}{1.5em}
\setcounter{section}{0}
%%%%%%%%%%%%%%%%%%%%%%%%%%%%%%%%%%%%%%%%%%%%%%%%%%

\Section{Introduction}

Let $0<\al\le2$, $T>0$ be constants and $\Om\subset\BR^d$ ($d=1,2,\ldots$) be a bounded domain with a smooth boundary $\pa\Om$. In this article, we consider the following initial-boundary value problem for a (time-fractional) evolution equation
\begin{equation}\label{eq-IBVP-u}
\begin{cases}
(\pa_{0+}^\al-\tri)u=F & \mbox{in }\Om\times(0,T),\\
\pa_t^k u=0\ (k=0,\lceil\al\rceil-1) & \mbox{in }\Om\times\{0\},\\
u=0 & \mbox{on }\pa\Om\times(0,T),
\end{cases}
\end{equation}
where the source term $F$ takes the form
\begin{equation}\label{eq-def-F}
F(\bm x,t):=\begin{cases}
f(\bm x-\bm p t), & 0<\al\le1,\\
f(\bm x-\bm p t)+g(\bm x-\bm q t), & 1<\al\le2.
\end{cases}
\end{equation}
In \eqref{eq-IBVP-u}, by $\tri:=\sum_{j=1}^d\f{\pa^2}{\pa x_j^2}$ we denote the usual Laplacian in space, and $\lceil\,\cdot\,\rceil$ is the ceiling function. The notation $\pa_{0+}^\al$ stands for the forward Caputo derivative in time of order $\alpha$, which will be defined in Section \ref{sec-prelim}. In \eqref{eq-def-F}, we assume that $\bm p,\bm q\in\BR^d$ are constant vectors, and assumptions on the space-dependent functions $f,g$ will be specified also in Section \ref{sec-prelim}.

As a continuation of the theoretical counterpart studied in \cite{LHY20}, this article is concerned with the establishment of numerical reconstruction schemes for the following inverse moving source problem regarding \eqref{eq-IBVP-u}--\eqref{eq-def-F}.

\begin{prob}[Inverse moving source problem]\label{prob-IMSP}
Let $u$ be the solution to {\rm\eqref{eq-IBVP-u}--\eqref{eq-def-F},} and $\om\subset\Om$ be a suitably chosen nonempty subdomain of $\Om$. Given constant vectors $\bm p,\bm q\in\BR^d$ such that $\bm p\ne\bm q,$ determine one source profile $f$ in the case of $0<\al\le1$ or two source profiles $f,g$ in the case of $1<\al\le2$ by the partial interior observation of $u$ in $\om\times(0,T)$.
\end{prob}

Except for the representative evolution equations of parabolic and hyperbolic types, the governing equation in \eqref{eq-IBVP-u} is called a time-fractional diffusion equation if $0<\al<1$ and a time-fractional diffusion-wave one if $1<\al<2$. In recent decades, time-fractional evolution equations have gathered increasing attention among applied mathematicians and, especially, general linear theories for \eqref{eq-IBVP-u} with $0<\al<1$ have been mostly established within the last decade (see, e.g., \cite{SY11a,GLY15,KY18}). Meanwhile, the corresponding numerical methods and inverse problems have also been studied intensively, and we refer e.g.\! to \cite{LX07,JLZ13} and \cite{JR15,LiuLiY19,LiLiuY19,LiY19} with the references therein, respectively. Compared with the case of $0<\al<1$, initial-boundary value problems such as \eqref{eq-IBVP-u} with $1<\al<2$ have not been well investigated from both theoretical and numerical aspects. On this direction, we refer to \cite{SY11a} for the well-posedness of forward problems, and \cite{LW19} as a recent progress on numerical approaches to related inverse problems.

Recently, it has been recognized that (time-fractional) evolution equations with orders $0<\al<2$ actually share several common properties, which can be applied to the qualitative analysis of some inverse problems. Besides the time-analyticity asserted in \cite{SY11a}, the weak vanishing property was established in \cite{JLLY17,LHY20} for $0<\al<1$ and $1<\al<2$ respectively, which resembles the classical unique continuation principle for parabolic equations. As direct applications, this property was employed to prove the uniqueness of an inverse $\bm x$-source problem in \cite{JLLY17} and that of Problem \ref{prob-IMSP} in \cite{LHY20}. On the other hand, since hyperbolic equations exhibit essentially different properties from the cases of $0<\al<2$, the uniqueness for Problem \ref{prob-IMSP} with the exceptional case of $\al=2$ requires special treatments and extra assumptions. The interested readers are referred to \cite{LHY20} for theoretical details of Problem \ref{prob-IMSP}, whose main result will be recalled in Section \ref{sec-prelim}.

By a moving source we refer to an inhomogeneous term in \eqref{eq-IBVP-u} basically taking the form of $h(\bm x-\bm\rho(t))$, where $h:\Om\longrightarrow\BR$ and $\bm\rho:[0,T]\longrightarrow\BR^d$ stand for the profile and the orbit of a moving object, respectively. In contrast to conventional inverse source problems, literature on inverse moving source problems seems rather limited, among which the majority focused on hyperbolic equations due to practical significance. We refer e.g.\! to \cite{NIO12,O19} for the reconstruction of moving point sources in wave equations, and \cite{HKLZ19} for the unique determination of profiles or orbits in Maxwell equations. Very recently, the stability for identifying orbits was obtained in \cite{HLY20} for (time-fractional) evolution equations such as \eqref{eq-IBVP-u} with $0<\al\le2$.

Parallel to \cite{HLY20}, the uniqueness for Problem \ref{prob-IMSP} was newly verified in \cite{LHY20}. Aiming at  possible real-world applications, this article is concerned with the development of efficient reconstruction schemes for Problem \ref{prob-IMSP}, which consists of three key ingredients. The first one coincides with that of \cite{LHY20}, that is, reducing the original problem to a backward problem in which the unknown profiles appear as initial conditions. The second one follows the same line as that in \cite{LJY15,JLY17a,JLY17b,JLLY17}, which utilized the iterative thresholding algorithm (see \cite{DDD04}) to deal with the corresponding minimization problem. Finally, for the simultaneous recovery of two profiles when $1<\al\le2$, the last step is to solve a convection equation, which is reduced to solving a series of second order ordinary differential equations.

The remainder of this article is organized as follows. In Section \ref{sec-prelim}, we prepare necessary tools of fractional calculus and recall theoretical facts regarding \eqref{eq-IBVP-u}--\eqref{eq-def-F} and Problem \ref{prob-IMSP}. The derivations of reconstruction schemes for one and two unknown profiles are explained in Sections \ref{sec-scheme1} and \ref{sec-scheme2}, respectively, and conclusions with future topics are given in Section \ref{sec-rem}.
%%%%%%%%%%%%%%%%%%%%%%%%%%%%%%%%%%%%%%%%%%%%%%%%%%

\Section{Preliminaries and revisit of theoretical results}\label{sec-prelim}

We start with fixing notations and terminologies. Throughout this article, all vectors are by default column vectors, e.g., $\bm x=(x_1,\ldots,x_d)^\T\in\BR^d$, where $(\,\cdot\,)^\T$ denotes the transpose. The inner product in $\BR^d$ is denoted by $\bm x\cdot\bm y$, and the Euclidean distance $|\cdot|$ is induced as $|\bm x|=(\bm x\cdot\bm x)^{1/2}$. By $L^2(\Om)$ we denote the space of square integrable functions in $\Om$, and let $H^k(\Om)$, $H_0^k(\Om)$, etc.\! and $W^{k,\ga}(0,T;H_0^1(\Om))$, etc.\! ($k=1,2,\ldots$, $\ga\in[1,\infty]$) be (vector-valued) Sobolev spaces (see e.g. \cite{A75,E98}).

For the definition of $\pa_{0+}^\al$ in \eqref{eq-IBVP-u}, we recall the forward Riemann-Liouville integral operator for $\be\in[0,1]$:
\[
J_{0+}^\be h(t):=\left\{\!\begin{alignedat}{2}
& h(t), & \quad & \be=0,\\
& \f1{\Ga(\be)}\int_0^t\f{h(\tau)}{(t-\tau)^{1-\be}}\,\rd\tau, & \quad & 0<\be\le1,
\end{alignedat}
\right.\quad h\in C[0,\infty),
\]
where $\Ga(\,\cdot\,)$ is the Gamma function. Then for $\be>0$, the forward Caputo derivative $\pa_{0+}^\be$ and the forward Riemann-Liouville derivative $D_{0+}^\be$ are formally defined as
\[
\pa_{0+}^\be=J_{0+}^{\lceil\be\rceil-\be}\circ\f{\rd^{\lceil\be\rceil}}{\rd t^{\lceil\be\rceil}},\quad D_{0+}^\be=\f{\rd^{\lceil\be\rceil}}{\rd t^{\lceil\be\rceil}}\circ J_{0+}^{\lceil\be\rceil-\be},
\]
where $\circ$ denotes the composition. Meanwhile, we also introduce the backward Riemann-Liouville integral operator as
\[
J_{T-}^\be h(t):=\left\{\!\begin{alignedat}{2}
& h(t), & \quad & \be=0,\\
& \f1{\Ga(\be)}\int_t^T\f{h(\tau)}{(\tau-t)^{1-\be}}\,\rd\tau, & \quad & 0<\be\le1,
\end{alignedat}\right.\quad h\in C(-\infty,T],
\]
and we further define the backward Caputo and Riemann-Liouville derivatives for $\be>0$ as
\[
\pa_{T-}^\be=J_{T-}^{\lceil\be\rceil-\be}\circ\f{\rd^{\lceil\be\rceil}}{\rd t^{\lceil\be\rceil}},\quad D_{T-}^\be=\f{\rd^{\lceil\be\rceil}}{\rd t^{\lceil\be\rceil}}\circ J_{T-}^{\lceil\be\rceil-\be}.
\]
For later use, we derive the useful fractional version of integration by parts.

\begin{lem}\label{lem-frac-int}
Let $h_1,h_2\in C^{\lceil\al\rceil}[0,T]$. If $0<\al\le1,$ then
\begin{equation}\label{eq-frac-int1}
\int_0^T(\pa_{0+}^\al h_1)\,h_2\,\rd t+h_1(0)(J_{T-}^{1-\al}h_2)(0)=-\int_0^T h_1\,(D_{T-}^\al h_2)\,\rd t+h_1(T)(J_{T-}^{1-\al}h_2)(T).
\end{equation}
If $1<\al\le2,$ then
\begin{align}
& \int_0^T(\pa_{0+}^\al h_1)\,h_2\,\rd t+h_1'(0)(J_{T-}^{2-\al}h_2)(0)=-\int_0^T h_1'\,(D_{T-}^{\al-1}h_2)\,\rd t+h_1'(T)(J_{T-}^{2-\al}h_2)(T),\label{eq-frac-int2a}\\
& \int_0^T h_1'\,(D_{T-}^{\al-1}h_2)\,\rd t+h_1(0)(D_{T-}^{\al-1}h_2)(0)=-\int_0^T h_1\,(D_{T-}^\al h_2)\,\rd t+h_1(T)(D_{T-}^{\al-1}h_2)(T).\label{eq-frac-int2b}
\end{align}
\end{lem}

\begin{proof}
We recall the following formula connecting the forward and backward Riemann-Liouville integral operators (see \cite[Lemma 4.1]{JLLY17})
\[
\int_0^T(J_{0+}^\be h_1)\,h_2\,\rd t=\int_0^T h_1\,(J_{T-}^\be h_2)\,\rd t,\quad0<\be\le1,\ h_1,h_2\in C[0,T],
\]
which can be easily verified by direct calculation. Then for $0<\al\le1$, we employ the above formula and the usual integration by parts to derive
\begin{align*}
\int_0^T(\pa_{0+}^\al h_1)\,h_2\,\rd t & =\int_0^T(J_{0+}^{1-\al}h_1')\,h_2\,\rd t=\int_0^T h_1'\,(J_{T-}^{1-\al}h_2)\,\rd t\\
& =\left[h_1\,(J_{T-}^{1-\al}h_2)\right]_0^T-\int_0^T h_1\,(J_{T-}^{1-\al}h_2)'\,\rd t\\
& =h_1(T)(J_{T-}^{1-\al}h_2)(T)-h_1(0)(J_{T-}^{1-\al}h_2)(0)-\int_0^T h_1\,(D_{T-}^\al h_2)\,\rd t
\end{align*}
for $h_1,h_2\in C^1[0,T]$, which gives \eqref{eq-frac-int1}. The derivations of \eqref{eq-frac-int2a}--\eqref{eq-frac-int2b} are similar and we omit the proofs here.
\end{proof}

Next, following the same line as \cite{LHY20}, we specify the key assumption on the source term \eqref{eq-def-F} and the observation subdomain $\om$ as
\begin{equation}\label{eq-asp}
\bigcup_{0\le t\le T}\{(\supp\,f+\bm p t)\cup(\supp\,g+\bm q t)\}\subset\subset\Om,\quad f,g\in H_0^{\lceil\al\rceil}(\Om),\quad\pa\om\supset\pa\Om.
\end{equation}
In other words, it is assumed that the source term $F$ defined in \eqref{eq-def-F} is compactly supported in $\Om\times[0,T]$ and is sufficiently smooth. Especially, we can suppose $f,g\in H_0^2(\Om)$ for $1<\al\le2$ since $f,g$ have compact supports in $\Om$. Here for simplicity, $\om$ is also assumed to surround the whole boundary of $\Om$ in the case of $\al=1$. For readers' better understanding, we illustrate the geometrical aspect of \eqref{eq-asp} in Figure \ref{fig-asp}.
\begin{figure}[htbp]\centering
\input{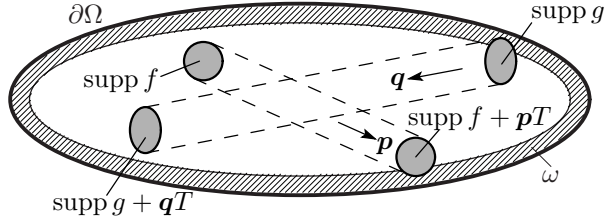}
\caption{An illustration of the geometrical assumptions on the moving source \eqref{eq-def-F} and the subdomain $\om$.}\label{fig-asp}
\end{figure}

Based on the above assumptions, now we state the well-posedness results concerning \eqref{eq-IBVP-u}--\eqref{eq-def-F}, which are slightly modified from \cite[Lemma 2.2]{LHY20} to fit into the framework of Sections \ref{sec-scheme1}--\ref{sec-scheme2}.

\begin{lem}\label{lem-IBVP-u}
Let $u$ be the solution to {\rm\eqref{eq-IBVP-u}--\eqref{eq-def-F},} where the profiles $f,g$ satisfy the key assumption \eqref{eq-asp}. Then the following statements hold true.

{\rm(a)} If $0<\al\le1,$ then $u\in L^\infty(0,T;H^2(\Om)\cap H_0^1(\Om))$ and $\pa_t u\in L^\ga(0,T;H_0^1(\Om)),$ where $\ga\in(1,\f1{1-\al})$ is arbitrarily fixed and we interpret $\f1{1-\al}=\infty$ for $\al=1$.

{\rm(b)} If $1<\al\le2,$ then $u\in L^\infty(0,T;H^3(\Om)\cap H_0^1(\Om)),$ $\pa_t u\in L^\infty(0,T;H^2(\Om)\cap H_0^1(\Om))$ and $\pa_t^2u\in L^\ga(0,T;H_0^1(\Om)),$ where $\ga\in(1,\f1{2-\al})$ is arbitrarily fixed and we interpret $\f1{2-\al}=\infty$ for $\al=2$.
\end{lem}

Compared with \cite[Lemma 2.2]{LHY20}, the time regularity of $\pa_t^{\lceil\al\rceil}u$ is improved in Lemma \ref{lem-IBVP-u}. Such an improvement is straightforward and we omit the detail here.

Finally, we close this section by recalling the uniqueness result for Problem \ref{prob-IMSP} obtained in \cite[Theorem 2.4]{LHY20}. Thanks to the linearity of Problem \ref{prob-IMSP}, it suffices to assume $u=0$ in $\om\times(0,T)$.

\begin{lem}\label{lem-unique}
Let $u$ be the solution to {\rm\eqref{eq-IBVP-u}--\eqref{eq-def-F},} where the profiles $f,g$ and the subdomain $\om$ satisfy the key assumption \eqref{eq-asp}.

{\rm(a)} If $0<\al\le1,$ then $u=0$ in $\om\times(0,T)$ implies $f\equiv0$ in $\Om$.

{\rm(b)} In the case of $1<\al\le2$, we additionally require $T>2\,\diam(\Om)$ if $\al=2$. Then $u=0$ in $\om\times(0,T)$ implies $f=g\equiv0$ in $\Om$.
\end{lem}
%%%%%%%%%%%%%%%%%%%%%%%%%%%%%%%%%%%%%%%%%%%%%%%%%%

\Section{Iterative thresholding scheme for one unknown profile}\label{sec-scheme1}

This section is devoted to the construction of a numerical inversion scheme for Problem \ref{prob-IMSP} on identifying a single profile via a minimization procedure. In the basic formulation \eqref{eq-def-F} and Problem \ref{prob-IMSP}, we assumed two profiles $f,g$ for $1<\al\le2$, which definitely includes the case and the determination of a single profile. Nevertheless, from a numerical viewpoint, the recovery of a single profile for $1<\al\le2$ also deserves consideration at least as a prototype of that of two profiles. Therefore, independent of the theoretical formulation, in this section we will consider the initial-boundary value problem
\begin{equation}\label{eq-IBVP-u1}
\begin{cases}
(\pa_{0+}^\al-\tri)u(\bm x,t)=f(\bm x-\bm p t), & (\bm x,t)\in\Om\times(0,T),\\
\pa_t^k u(\bm x,0)=0\ (k=0,\lceil\al\rceil-1), & \bm x\in\Om,\\
u(\bm x,t)=0, & (\bm x,t)\in\pa\Om\times(0,T),
\end{cases}
\end{equation}
where $0<\al\le2$ and $f\in H_0^1(\Om)$. To emphasize the dependence, we will denote the solution to \eqref{eq-IBVP-u1} by $u(f)$.

Let $f_\true\in H_0^1(\Om)$ be the true profile and $u^\de$ be the observation data. For technical convenience, we assume that there exists a $\ga\in(1,\f1{\lceil\al\rceil-\al})$ such that
\begin{equation}\label{eq-asp-data}
u^\de\in W^{1,\ga}(0,T;H^1(\om)),\quad\|u(f_\true)-u^\de\|_{W^{1,\ga}(0,T;H^1(\om)))}\le\de,
\end{equation}
which means that the noise in the observation data is bounded by $\de$ in $W^{1,\ga}(0,T;H^1(\om))$. Although \eqref{eq-asp-data} looks restrictive for observation data, it is essential for the establishment of the inversion scheme in the sequel. In view of Lemma \ref{lem-IBVP-u}, it turns out that \eqref{eq-asp-data} is tolerable. On the other hand, it suffices to add a preconditioning procedure to mollify the noisy data by some stable numerical differentiation method (see e.g. \cite{WJC02}).

\Subsection{Case of $0<\al\le1$}

First we consider the case of $0<\al\le1$. Similarly to the proof of \cite[Theorem 2.4]{LHY20}, we introduce an auxiliary function $v(f):=J_{0+}^{1-\al}(\pa_t+\bm p\cdot\nb)u(f)$, which satisfies the following initial-boundary value problem for a (time-fractional) diffusion equation with a Caputo derivative in time:
\[
\begin{cases}
(\pa_{0+}^\al-\tri)v=0 & \mbox{in }\Om\times(0,T),\\
v=f & \mbox{in }\Om\times\{0\},\\
v=J_{0+}^{1-\al}(\bm p\cdot\nb u(f)) & \mbox{on }\pa\Om\times(0,T).
\end{cases}
\]
Here the boundary condition of $v(f)$ follows from that of $u(f)$ and the assumption $\pa\om\supset\pa\Om$. Especially, on $\pa\Om\times(0,T)$, we know $v(f_\true)=J_{0+}^{1-\al}(\bm p\cdot\nb u(f_\true))$ is approximated by $J_{0+}^{1-\al}(\bm p\cdot\nb u^\de)$. Then for any $f\in H_0^1(\Om)$, we can freeze the boundary condition of $v(f)$ as $J_{0+}^{1-\al}(\bm p\cdot\nb u^\de)|_{\pa\Om\times(0,T)}$, i.e., $v(f)$ satisfies
\begin{equation}\label{eq-IBVP-v}
\begin{cases}
(\pa_{0+}^\al-\tri)v=0 & \mbox{in }\Om\times(0,T),\\
v=f & \mbox{in }\Om\times\{0\},\\
v=J_{0+}^{1-\al}(\bm p\cdot\nb u^\de) & \mbox{on }\pa\Om\times(0,T).
\end{cases}
\end{equation}
According to Lemma \ref{lem-IBVP-u}(a) and H\"older's inequality, it is readily seen that $v(f)\in L^\infty(0,T;H^1(\Om))$ for any $f\in H_0^1(\Om)$.

Now we set $v^\de:=J_{0+}^{1-\al}(\pa_t+\bm p\cdot\nb)u^\de\in L^\infty(0,T;L^2(\om))$ and consider the following regularized minimization problem with a penalty term
\begin{equation}\label{eq-def-Phi}
\min_{f\in H_0^1(\Om)}\Phi(f):=\left\|v(f)-v^\de\right\|_{L^2(\om\times(0,T))}^2+\ka\|\nb f\|_{L^2(\Om)}^2,
\end{equation}
where $\ka>0$ is the regularization parameter. Here we notice that Poincar\'e inequality guarantees the equivalence of $\|f\|_{H^1(\Om)}$ and $\|\nb f\|_{L^2(\Om)}$. Next, we calculate the Fr\'echet derivative of $\Phi(f)$ for $f\in H_0^1(\Om)$ on the direction $\wt f\in H_0^1(\Om)$. Picking a small $\ve>0$, we calculate
\begin{align*}
\f{\Phi(f+\ve\,\wt f\,)-\Phi(f)}\ve & =\int_0^T\!\!\!\int_\om\f{v(f+\ve\,\wt f\,)-v(f)}\ve\left(v(f+\ve\,\wt f\,)+v(f)-2\,v^\de\right)\rd\bm x\rd t\\
& \quad\,+\ka\int_\Om\nb\wt f\cdot(2\nb f+\ve\nb\wt f\,)\,\rd\bm x.
\end{align*}
Taking difference between the initial-boundary value problems of $v(f+\ve\,\wt f\,)$ and $v(f)$, we see that a further auxiliary function $w(\wt f\,):=(v(f+\ve\,\wt f\,)-v(f))/\ve$ satisfies
\begin{equation}\label{eq-IBVP-w0}
\begin{cases}
(\pa_{0+}^\al-\tri)w=0 & \mbox{in }\Om\times(0,T),\\
w=\wt f & \mbox{in }\Om\times\{0\},\\
w=0 & \mbox{on }\pa\Om\times(0,T).
\end{cases}
\end{equation}
By linearity, we have $w(\wt f\,)\in L^\infty(0,T;H_0^1(\Om))$. Then we deduce
\begin{align}
\f{\Phi'(f)\wt f}2 & =\lim_{\ve\to0}\f{\Phi(f+\ve\,\wt f\,)-\Phi(f)}{2\ve}\nonumber\\
& =\f12\int_0^T\!\!\!\int_\om w(\wt f\,)\left(\lim_{\ve\to0}v(f+\ve\,\wt f\,)+v(f)-2\,v^\de\right)\rd\bm x\rd t+\ka\int_\Om\nb\wt f\cdot\nb f\,\rd\bm x\nonumber\\
& =\int_0^T\!\!\!\int_\om w(\wt f\,)\left(v(f)-v^\de\right)\rd\bm x\rd t-\ka\int_\Om\wt f\,\tri f\,\rd\bm x,\label{eq-Frechet0}
\end{align}
where we observed $v(f+\ve\,\wt f\,)\longrightarrow v(f)$ in $L^1(0,T;L^2(\om))$ as $\ve\to0$ by the continuous dependence of the solution to \eqref{eq-IBVP-v} upon the initial value. Here we understand $\tri f\in H^{-1}(\Om)$ and that the inner product $\int_\Om\wt f\,\tri f\,\rd\bm x$ stands for the duality pairing ${}_{H^{-1}}\langle\tri f,\wt f\,\rangle_{H_0^1}$.

Now it remains to derive the explicit form of $\Phi'(f)$. To this end, we introduce the backward problem
\begin{equation}\label{eq-IBVP-y0}
\begin{cases}
(D_{T-}^\al+\tri)y=-\chi_\om\left(v(f)-v^\de\right) & \mbox{in }\Om\times(0,T),\\
J_{T-}^{1-\al}y=0 & \mbox{in }\Om\times\{T\},\\
y=0 & \mbox{on }\pa\Om\times(0,T),
\end{cases}
\end{equation}
where $\chi_\om$ is the characteristic function of $\om$. As before, we denote the solution to \eqref{eq-IBVP-y0} as $y(f)$ to emphasize its dependence on $f$. For the theoretical analysis and the numerical solution of \eqref{eq-IBVP-y0} in the case of $0<\al<1$, one difficulty is the treatment of the terminal condition, which involves the backward Riemann-Liouville integral as $t\to T$. To circumvent this, we further introduce $z(f):=J_{T-}^{1-\al}y(f)$, which obviously satisfies
\begin{equation}\label{eq-IBVP-z}
\begin{cases}
(\pa_{T-}^\al-\tri)z=-\chi_\om\,J_{T-}^{1-\al}\left(v(f)-v^\de\right) & \mbox{in }\Om\times(0,T),\\
z=0 & \mbox{in }\Om\times\{T\},\\
z=0 & \mbox{on }\pa\Om\times(0,T).
\end{cases}
\end{equation}
Since $v(f)-v^\de\in L^\infty(0,T;L^2(\om))$, the source term of \eqref{eq-IBVP-z} belongs to $L^\ga(0,T;L^2(\Om))$ for any $\ga\in(1,\f1{1-\al})$. Similarly to the proof of Lemma \ref{lem-IBVP-u}(a), one can verify $\pa_t z(f)\in L^\ga(0,T;H_0^1(\Om))$. Hence, differentiating $z(f)=J_{T-}^{1-\al}y(f)$ indicates $D_{T-}^\al y(f)=\pa_t z(f)\in L^\ga(0,T;H_0^1(\Om))$ for any $\ga\in(1,\f1{1-\al})$.

To proceed, we shall turn to the variational method to define the weak solutions to \eqref{eq-IBVP-w0} and \eqref{eq-IBVP-y0}. Motivated by the definition of weak solutions to traditional evolution equations e.g.\! in \cite{E98}, we employ formula \eqref{eq-frac-int1} in Lemma \ref{lem-frac-int} to provide the following definition.

\begin{defi}\label{def-weak-sol0}
Let $0<\al\le1$ and $f,\wt f\in H_0^1(\Om)$.

{\rm(a)} We say that $w(\wt f\,)\in L^\infty(0,T;H_0^1(\Om))$ is a weak solution to \eqref{eq-IBVP-w0} if
\[
\int_0^T\!\!\!\int_\Om\left(\nb w(\wt f\,)\cdot\nb y-w(\wt f\,)\,(D_{T-}^\al y)\right)\rd\bm x\rd t=\int_\Om\wt f\,(J_{T-}^{1-\al}y)(\,\cdot\,,0)\,\rd\bm x
\]
holds for all test functions $y\in L^1(0,T;H_0^1(\Om))$ satisfying $D_{T-}^\al y\in L^1(0,T;L^2(\Om))$ and $J_{T-}^{1-\al}y=0$ in $\Om\times\{T\}$.

{\rm(b)} For some fixed $\ga\in(1,\f1{1-\al}),$ we say that $y(f)\in L^\ga(0,T;H_0^1(\Om))$ satisfying $D_{T-}^\al y(f)\in L^\ga(0,T;L^2(\Om))$ is a weak solution to \eqref{eq-IBVP-y0} if $J_{T-}^{1-\al}y(f)=0$ in $\Om\times\{T\}$ and
\[
\int_0^T\!\!\!\int_\Om\left(\nb w\cdot\nb y(f)-w\,(D_{T-}^\al y(f))\right)\rd\bm x\rd t=\int_0^T\!\!\!\int_\om w\left(v(f)-v^\de\right)\rd\bm x\rd t
\]
holds for all test functions $w\in L^\infty(0,T;H_0^1(\Om))$.
\end{defi}

In the above definition, the weak solution to the backward problem \eqref{eq-IBVP-y0} is slightly stronger than that to the forward problem \eqref{eq-IBVP-w0} because $D_{T-}^\al y(f)$ makes sense in $L^\ga(0,T;L^2(\Om))$. Now we can take $y=y(f)$ and $w=w(\wt f\,)$ as mutual test functions in Definition \ref{def-weak-sol0} to calculate
\begin{align*}
\int_0^T\!\!\!\int_\om w(\wt f\,)\left(v(f)-v^\de\right)\rd\bm x\rd t & =\int_0^T\!\!\!\int_\Om\left(\nb w(\wt f\,)\cdot\nb y(f)-w(\wt f\,)\,(D_{T-}^\al y(f))\right)\rd\bm x\rd t\\
& =\int_\Om\wt f\,(J_{T-}^{1-\al}y(f))(\,\cdot\,,0)\,\rd\bm x=\int_\Om\wt f\,z(f)(\,\cdot\,,0)\,\rd\bm x.
\end{align*}
Plugging the above equality into \eqref{eq-Frechet0}, we obtain
\[
\f{\Phi'(f)\wt f}2=\int_\Om\wt f\,(z(f)(\,\cdot\,,0)-\ka\,\tri f)\,\rd\bm x.
\]
Since $\wt f\in H_0^1(\Om)$ is chosen arbitrarily, it follows from the variational principle that the minimizer $f_*\in H_0^1(\Om)$ of \eqref{eq-def-Phi} is the weak solution to the following boundary value problem for an elliptic equation
\begin{equation}\label{eq-Euler-Lagrange}
\begin{cases}
\ka\,\tri f_*=z(f_*)(\,\cdot\,,0) & \mbox{in }\Om,\\
f_*=0 & \mbox{on }\pa\Om.
\end{cases}
\end{equation}
Consequently, in the same manner as \cite{JLLY17}, we can propose the following iterative thresholding update
\begin{equation}\label{eq-BVP-f}
\left\{\!\begin{alignedat}{2}
& \tri f_{\ell+1}=\f1{M+\ka}z(f_\ell)(\,\cdot\,,0)+\f M{M+\ka}\tri f_\ell & \quad & \mbox{in }\Om,\\
& f_{\ell+1}=0 & \quad & \mbox{on }\pa\Om,
\end{alignedat}\right.
\end{equation}
where $M>0$ is a tuning parameter. Now we summarize the numerical reconstruction scheme for Problem \ref{prob-IMSP} with $0<\al\le1$ as follows.

\begin{algo}\label{algo-1}
Choose a tolerance $\ep>0$, a regularization parameter $\ka>0$ and a tuning parameter $M>0$. Give an initial guess $f_0\in H_0^1(\Om)$ (e.g., $f_0\equiv0$), and set $\ell=0$.
\begin{enumerate}
\item Compute $f_{\ell+1}$ by the iterative update \eqref{eq-BVP-f}.
\item If $\|f_{\ell+1}-f_\ell\|_{H^1(\Om)}/\|f_\ell\|_{H^1(\Om)}<\ep$, stop the iteration. Otherwise, update $\ell\leftarrow\ell+1$ and return to Step 1.
\end{enumerate}
\end{algo}

By choosing $M>0$ suitably large, the convergence of the sequence $\{f_\ell\}_{\ell=0}^\infty\subset H_0^1(\Om)$ is guaranteed by \cite{DDD04}. In each step of the iteration, it suffices to solve $3$ differential equations, i.e., the modified forward problem \eqref{eq-IBVP-v} for $v(f_\ell)$, the backward problem \eqref{eq-IBVP-z} for $z(f_\ell)$, and the boundary value problem for the Poisson equation \eqref{eq-BVP-f} for $f_{\ell+1}$.

\begin{rem}
Similarly, if we penalize the $L^2$-norm of $f$ instead of $\nb f$ in the definition \eqref{eq-def-Phi} of $\Phi(f)$, then the resulting iterative thresholding update turns out to be
\[
f_{\ell+1}=\f1{M+\ka}z(f_\ell)(\,\cdot\,,0)+\f M{M+\ka}f_\ell,
\]
which differs from the scheme \eqref{eq-BVP-f} only on the cost of solving a Poisson equation numerically. Nevertheless, since we assumed $f\in H_0^1(\Om)$, such sacrifice on the computational cost is unimportant in view of pursuing more accurate reconstruction of the profile in the $H^1$ sense.
\end{rem}

\Subsection{Case of $1<\al\le2$}

Now we consider \eqref{eq-IBVP-u1} in the case of $1<\al\le2$ with $f\in H_0^1(\Om)$. Analogously to the proof of Lemma \ref{lem-IBVP-u}(b), one can easily check
\[
u(f)\in L^\infty(0,T;H^2(\Om)\cap H_0^1(\Om)),\quad\pa_t u(f)\in L^\ga(0,T;H_0^1(\Om)),
\]
where $\ga\in(1,\f1{2-\al})$. Again setting $v(f):=J_{0+}^{2-\al}(\pa_t+\bm p\cdot\nb)u(f)$, $v^\de:=J_{0+}^{2-\al}(\pa_t+\bm p\cdot\nb)u^\de$ and freezing the boundary condition of $v(f)$ as $J_{0+}^{2-\al}(\bm p\cdot\nb u^\de)$, we know $v(f)\in W^{1,\infty}(0,T;H^1(\Om))$, $v^\de\in L^\infty(0,T;L^2(\om))$ and $v(f)$ satisfies
\[
\begin{cases}
(\pa_{0+}^\al-\tri)v=0 & \mbox{in }\Om\times(0,T),\\
v=0,\ \pa_t v=f & \mbox{in }\Om\times\{0\},\\
v=J_{0+}^{2-\al}(\bm p\cdot\nb u^\de) & \mbox{on }\pa\Om\times(0,T).
\end{cases}
\]
Treating the same minimization problem \eqref{eq-def-Phi} and repeating the calculation of the Fr\'echet derivative, again we obtain \eqref{eq-Frechet0}, where $w(\wt f)\in W^{1,\infty}(0,T;H_0^1(\Om))$ satisfies
\begin{equation}\label{eq-IBVP-w1}
\begin{cases}
(\pa_{0+}^\al-\tri)w=0 & \mbox{in }\Om\times(0,T),\\
w=0,\ \pa_t w=\wt f & \mbox{in }\Om\times\{0\},\\
w=0 & \mbox{on }\pa\Om\times(0,T).
\end{cases}
\end{equation}

In a parallel manner as before, we investigate the solution $y(f)$ to the backward problem
\begin{equation}\label{eq-IBVP-y1}
\begin{cases}
(D_{T-}^\al-\tri)y=\chi_\om\left(v(f)-v^\de\right) & \mbox{in }\Om\times(0,T),\\
J_{T-}^{2-\al}y=D_{T-}^{\al-1}y=0 & \mbox{in }\Om\times\{T\},\\
y=0 & \mbox{on }\pa\Om\times(0,T)
\end{cases}
\end{equation}
and the further auxiliary function $z(f):=J_{T-}^{2-\al}y(f)$ satisfying
\[
\begin{cases}
(\pa_{T-}^\al-\tri)z=\chi_\om\,J_{T-}^{2-\al}\left(v(f)-v^\de\right) & \mbox{in }\Om\times(0,T),\\
z=\pa_t z=0 & \mbox{in }\Om\times\{T\},\\
z=0 & \mbox{on }\pa\Om\times(0,T).
\end{cases}
\]
By a similar argument as that in the previous subsection, one can verify $y(f),D_{T-}^{\al-1}y(f)\in L^\ga(0,T;H_0^1(\Om))$ with some $\ga\in(1,\f1{2-\al})$. Next, we employ formulae \eqref{eq-frac-int2a} and \eqref{eq-frac-int2b} in Lemma \ref{lem-frac-int}(b) to define the weak solutions to \eqref{eq-IBVP-w1} and \eqref{eq-IBVP-y1}, respectively.

\begin{defi}\label{def-weak-sol1}
Let $1<\al\le2$ and $f,\wt f\in H_0^1(\Om)$.

{\rm(a)} We say that $w(\wt f\,)\in L^\infty(0,T;H_0^1(\Om))\cap W^{1,\infty}(0,T;L^2(\Om))$ is a weak solution to \eqref{eq-IBVP-w1} if
\[
\int_0^T\!\!\!\int_\Om\left\{\nb w(\wt f\,)\cdot\nb y-\left(\pa_t w(\wt f\,)\right)(D_{T-}^{\al-1}y)\right\}\rd\bm x\rd t=\int_\Om\wt f\,(J_{T-}^{2-\al}y)(\,\cdot\,,0)\,\rd\bm x
\]
holds for all test functions $y\in L^1(0,T;H_0^1(\Om))$ satisfying $D_{T-}^{\al-1}y\in L^1(0,T;L^2(\Om))$ and $J_{T-}^{2-\al}y=0$ in $\Om\times\{T\}$.

{\rm(b)} For some fixed $\ga\in(1,\f1{2-\al}),$ we say that $y(f)\in L^\ga(0,T;H_0^1(\Om))$ satisfying $D_{T-}^{\al-1}y(f)\in L^\ga(0,T;L^2(\Om))$ is a weak solution to \eqref{eq-IBVP-y1} if $J_{T-}^{2-\al}y(f)=D_{T-}^{\al-1}y(f)=0$ in $\Om\times\{T\}$ and
\[
\int_0^T\!\!\!\int_\Om\left(\nb w\cdot\nb y(f)-(\pa_t w)\,(D_{T-}^{\al-1}y(f))\right)\rd\bm x\rd t=\int_0^T\!\!\!\int_\om w\left(v(f)-v^\de\right)\rd\bm x\rd t
\]
holds for all test functions $w\in L^\infty(0,T;H_0^1(\Om))\cap W^{1,\infty}(0,T;L^2(\Om))$ satisfying $w(f)=0$ in $\Om\times\{0\}$.
\end{defi}

Taking $w(\wt f\,)$ and $y(f)$ as mutual test functions in the above definition, again we obtain
\begin{align*}
\int_0^T\!\!\!\int_\om w(\wt f\,)\left(v(f)-v^\de\right)\rd\bm x\rd t & =\int_0^T\!\!\!\int_\Om\left\{\nb w(\wt f\,)\cdot\nb y(f)-\left(\pa_t w(\wt f\,)\right)(D_{T-}^{\al-1}y(f))\right\}\rd\bm x\rd t\\
& =\int_\Om\wt f\,(J_{T-}^{2-\al}y(f))(\,\cdot\,,0)\,\rd\bm x=\int_\Om\wt f\,z(f)(\,\cdot\,,0)\,\rd\bm x.
\end{align*}
As a result, formally we arrive at the same Euler-Lagrange equation \eqref{eq-Euler-Lagrange} and thus the thresholding iterative update \eqref{eq-BVP-f}.
%%%%%%%%%%%%%%%%%%%%%%%%%%%%%%%%%%%%%%%%%%%%%%%%%%

\Section{Reconstruction scheme for two unknown profiles}\label{sec-scheme2}

On the same direction as that of the previous section, in this section we attempt to develop a numerical reconstruction scheme for Problem \ref{prob-IMSP} in the case of two unknown profiles. More precisely, let $1<\al\le2$, $f,g\in H_0^2(\Om)$ and denote by $u(f,g)$ the solution to \eqref{eq-IBVP-u}--\eqref{eq-def-F}. Similarly as before, by $f_\true,g_\true\in H_0^2(\Om)$ and $u^\de$ we denote the true profiles and the observation data, respectively. As a generalization of \eqref{eq-asp-data}, here we assume that there exists $\ga\in(1,\f1{2-\al})$ such that
\[
u^\de\in W^{2,\ga}(0,T;H^2(\om)),\quad\|u(f_\true,g_\true)-u^\de\|_{W^{2,\ga}(0,T;H^2(\om))}\le\de.
\]
Meanwhile, introducing
\[
v(f,g):=J_{0+}^{2-\al}(\pa_t+\bm p\cdot\nb)(\pa_t+\bm q\cdot\nb)u(f,g),\quad v^\de:=J_{0+}^{2-\al}(\pa_t+\bm p\cdot\nb)(\pa_t+\bm q\cdot\nb)u^\de
\]
and freezing the boundary condition of $v(f,g)$ as
\[
v^\de=(\bm p+\bm q)\cdot\nb(\pa_{0+}^{\al-1}u^\de)+\bm p\cdot\nb(\bm q\cdot\nb(J_{0+}^{2-\al}u^\de))\quad\mbox{on }\pa\Om\times(0,T),
\]
we see that $v^\de\in L^\infty(0,T;L^2(\om))$ and it was shown in \cite[\S4.2]{LHY20} that $v(f,g)$ satisfies
\begin{equation}\label{eq-IBVP-v2}
\begin{cases}
(\pa_{0+}^\al-\tri)v=0 & \mbox{in }\Om\times(0,T),\\
v=f+g,\ \pa_t v=\bm q\cdot\nb f+\bm p\cdot\nb g & \mbox{in }\Om\times\{0\},\\
v=v^\de & \mbox{on }\pa\Om\times(0,T).
\end{cases}
\end{equation}
Then it follows from Lemma \ref{lem-IBVP-u}(b) and H\"older's inequality that $v(f,g)\in L^\infty(0,T;H^1(\Om))$ for any $f,g\in H_0^2(\Om)$.

Since the unknown profiles $f,g$ appears in the initial values of \eqref{eq-IBVP-v2}, it is natural to divide the reconstruction into two steps, i.e., the simultaneous recovery of the initial values and then the determination of $f,g$ by the information of $f+g$ and $\bm q\cdot\nb f+\bm p\cdot\nb g$. To this end, it is convenient to rewrite the initial condition of \eqref{eq-IBVP-v2} as
\begin{equation}\label{eq-ab-fg}
v(\,\cdot\,,0)=a:=f+g\in H_0^2(\Om),\quad \pa_t v(\,\cdot\,,0)=b:=\bm q\cdot\nb f+\bm p\cdot\nb g\in H_0^1(\Om)
\end{equation}
and denote the solution to \eqref{eq-IBVP-v2} as $v(a,b)$.

\Subsection{Iterative thresholding scheme for the initial values}

As a natural generalization of the minimization problem \eqref{eq-def-Phi}, we consider the following multivariable minimization problem with two penalty terms
\begin{equation}\label{eq-def-Psi}
\min_{(a,b)\in H_0^2(\Om)\times H_0^1(\Om)}\Psi(a,b):=\left\|v(a,b)-v^\de\right\|_{L^2(\om\times(0,T))}^2+\ka_2\|\tri a\|_{L^2(\Om)}^2+\ka_1\|\nb b\|_{L^2(\Om)}^2,
\end{equation}
where $\ka_1,\ka_2>0$ are regularization parameters. Here, owing to $a\in H_0^2(\Om)$ and $b\in H_0^1(\Om)$, we know that $\|\tri a\|_{L^2(\Om)}$ and $\|\nb b\|_{L^2(\Om)}$ are equivalent to $\|a\|_{H^2(\Om)}$ and $\|b\|_{H^1(\Om)}$, respectively.

To calculate the Fr\'echet derivative of $\Psi$ for $(a,b)\in H_0^2(\Om)\times H_0^1(\Om)$ on the direction $(\wt a,\wt b\,)\in H_0^2(\Om)\times H_0^1(\Om)$, we pick a small $\ve>0$ to calculate
\begin{align*}
\f{\Psi(a+\ve\,\wt a,b+\ve\,\wt b\,)-\Psi(a,b)}\ve & =\int_0^T\!\!\!\int_\om w(\wt a,\wt b\,)\left\{v(a+\ve\,\wt a,b+\ve\,\wt b\,)+v(a,b)-2\,v^\de\right\}\rd\bm x\rd t\\
& \quad\,+\ka_2\int_\Om\tri\wt a\,\tri(2a+\ve\,\wt a\,)\,\rd\bm x+\ka_1\int_\Om\nb\wt b\cdot\nb(2b+\ve\,\wt b\,)\,\rd\bm x,
\end{align*}
where $w(\wt a,\wt b\,):=(v(a+\ve\,\wt a,b+\ve\,\wt b\,)-v(a,b))/\ve\in L^\infty(0,T;H_0^1(\Om))$ satisfies
\begin{equation}\label{eq-IBVP-w2}
\begin{cases}
(\pa_{0+}^\al-\tri)w=0 & \mbox{in }\Om\times(0,T),\\
w=\wt a,\ \pa_t w=\wt b & \mbox{in }\Om\times\{0\},\\
w=0 & \mbox{on }\pa\Om\times(0,T).
\end{cases}
\end{equation}
Then by passing $\ve\to0$ and integration by parts, we deduce
\begin{align}
\f{\nb\Psi(a,b)\cdot(\wt a,\wt b\,)}2 & =\lim_{\ve\to0}\f{\Psi(a+\ve\,\wt a,b+\ve\,\wt b\,)-\Psi(a,b)}{2\ve}\nonumber\\
& =\int_0^T\!\!\!\int_\om w(\wt a,\wt b\,)\left(v(a,b)-v^\de\right)\rd\bm x\rd t+\ka_2\int_\Om\tri\wt a\,\tri a\,\rd\bm x+\ka_1\int_\Om\nb\wt b\cdot\nb b\,\rd\bm x\nonumber\\
& =\int_0^T\!\!\!\int_\Om w(\wt a,\wt b\,)\left\{\chi_\om\left(v(a,b)-v^\de\right)\right\}\rd\bm x\rd t+\int_\Om\left(\ka_2\wt a\,\tri^2a-\ka_1\wt b\,\tri b\right)\rd\bm x\nonumber\\
& =\int_\Om\left[\int_0^T w(\wt a,\wt b\,)\left\{\chi_\om\left(v(a,b)-v^\de\right)\right\}\rd t+\begin{pmatrix}
\wt a\,\\
\wt b
\end{pmatrix}\cdot\begin{pmatrix}
\ka_2\tri^2a\\
-\ka_1\tri b
\end{pmatrix}\right]\rd\bm x,\label{eq-Frechet1}
\end{align}
where we employed $v(a+\ve\,\wt a,b+\be\,\wt b\,)\longrightarrow v(a,b)$ in $L^1(0,T;L^2(\om))$ as $\ve\to0$, and interpret $\tri^2a\in H^{-2}(\Om)$ and $\tri b\in H^{-1}(\Om)$ for convenience.

In order to derive the Euler-Lagrange equation for the minimizer of \eqref{eq-def-Psi}, again we shall turn to the solution $y(a,b)$ to the backward problem
\begin{equation}\label{eq-IBVP-y2}
\begin{cases}
(D_{T-}^\al-\tri)y=\chi_\om\left(v(a,b)-v^\de\right) & \mbox{in }\Om\times(0,T),\\
J_{T-}^{2-\al}y=D_{T-}^{\al-1}y=0 & \mbox{in }\Om\times\{T\},\\
y=0 & \mbox{on }\pa\Om\times(0,T)
\end{cases}
\end{equation}
as well as its Caputo counterpart $z(a,b):=J_{T-}^{2-\al}y(a,b)$ satisfying
\begin{equation}\label{eq-ibvp-z2}
\begin{cases}
(\pa_{T-}^\al-\tri)z=\chi_\om\,J_{T-}^{2-\al}\left(v(a,b)-v^\de\right) & \mbox{in }\Om\times(0,T),\\
z=\pa_t z=0 & \mbox{in }\Om\times\{T\},\\
z=0 & \mbox{on }\pa\Om\times(0,T).
\end{cases}
\end{equation}
Similarly to the argument for \eqref{eq-IBVP-y0} and \eqref{eq-IBVP-z}, one can employ Lemma \ref{lem-IBVP-u}(b) to verify that
\[
y(a,b)=\pa_{T-}^{\al-1}z(a,b)\in L^\ga(0,T;H_0^1(\Om)),\quad D_{T-}^\al y(a,b)=\pa_t^2z(a,b)\in L^\ga(0,T,L^2(\Om))
\]
for any $\ga\in(1,\f1{2-\al})$.

Now we are well prepared to provide the definitions of the weak solutions to \eqref{eq-IBVP-w2} and \eqref{eq-IBVP-y2}.

\begin{defi}\label{def-weak-sol2}
Let $1<\al\le2$ and $a,b,\wt a,\wt b\in H_0^2(\Om)$.

{\rm(a)} We say that $w(\wt a,\wt b\,)\in L^\infty(0,T;H_0^1(\Om))$ is a weak solution to \eqref{eq-IBVP-w2} if
\[
\int_0^T\!\!\!\int_\Om\left(\nb w(\wt a,\wt b\,)\cdot\nb y+w(\wt a,\wt b\,)\,(D_{T-}^\al y)\right)\rd\bm x\rd t=\int_\Om\left(\wt b\,J_{T-}^{2-\al}y-\wt a\,D_{T-}^{\al-1}y\right)(\,\cdot\,,0)\,\rd\bm x
\]
holds for all test functions $y\in L^1(0,T;H_0^1(\Om))$ satisfying $D_{T-}^\al y\in L^1(0,T;L^2(\Om))$ and $J_{T-}^{2-\al}y=D_{T-}^{\al-1}y=0$ in $\Om\times\{T\}$.

{\rm(b)} For some fixed $\ga\in(1,\f1{2-\al}),$ we say that $y(a,b)\in L^\ga(0,T;H_0^1(\Om))$ satisfying $D_{T-}^\al y(a,b)\in L^\ga(0,T;L^2(\Om))$ is a weak solution to \eqref{eq-IBVP-y2} if $J_{T-}^{2-\al}y(a,b)=D_{T-}^{\al-1}y(a,b)=0$ in $\Om\times\{T\}$ and
\[
\int_0^T\!\!\!\int_\Om\left(\nb w\cdot\nb y(a,b)+w\,(D_{T-}^\al y(a,b))\right)\rd\bm x\rd t=\int_0^T\!\!\!\int_\om w\left(v(a,b)-v^\de\right)\rd\bm x\rd t
\]
holds for all test functions $w\in L^\infty(0,T;H_0^1(\Om))$.
\end{defi}

In comparison with Definition \ref{def-weak-sol1}, the above definition requires lower time regularity for the solution to the forward problem \eqref{eq-IBVP-w2}, whereas requires higher time regularity for that to the backward problem \eqref{eq-IBVP-y2}. As before, it is routine to take $y(a,b)$ and $w(\wt a,\wt b\,)$ as mutual test functions to deduce
\begin{align*}
\int_0^T\!\!\!\int_\om w(\wt a,\wt b\,)\left(v(a,b)-v^\de\right)\rd\bm x\rd t & =\int_0^T\!\!\!\int_\Om\left(\nb w(\wt a,\wt b\,)\cdot\nb y(a,b)+w(\wt a,\wt b\,)\,(D_{T-}^\al y(a,b))\right)\rd\bm x\rd t\\
& =\int_\Om\left(\wt b\,J_{T-}^{2-\al}y(a,b)-\wt a\,D_{T-}^{\al-1}y(a,b)\right)(\,\cdot\,,0)\,\rd\bm x\\
& =\int_\Om\begin{pmatrix}
\wt a\,\\
\wt b
\end{pmatrix}\cdot\begin{pmatrix}
-\pa_t z(a,b)\\
z(a,b)
\end{pmatrix}(\,\cdot\,,0)\,\rd\bm x.
\end{align*}
Substituting the above equality into \eqref{eq-Frechet1}, we arrive at
\[
\f{\nb\Psi(a,b)\cdot(\wt a,\wt b\,)}2=\int_\Om\begin{pmatrix}
\wt a\,\\
\wt b
\end{pmatrix}\cdot\begin{pmatrix}
-\pa_t z(a,b)(\,\cdot\,,0)+\ka_2\tri^2a\\
z(a,b)(\,\cdot\,,0)-\ka_1\tri b
\end{pmatrix}\rd\bm x.
\]
In other words, the minimizer $(a_*,b_*)\in H_0^2(\Om)\times H_0^1(\Om)$ of \eqref{eq-def-Psi} is the weak solution to the following Euler-Lagrange equation
\[
\begin{cases}
\ka_2\tri^2a_*=\pa_t z(a_*,b_*)(\,\cdot\,,0),\\
\ka_1\tri b_*=z(a_*,b_*)(\,\cdot\,,0)
\end{cases}\mbox{in }\Om.
\]
Especially, $a_*$ satisfies a bi-Laplace equation in the sense of $H^{-2}(\Om)$. Therefore, in a similar manner as before, we can design the following iterative thresholding update to produce a sequence $\{(a_\ell,b_\ell)\}\subset H_0^2(\Om)\times H_0^1(\Om)$ approximating $(a_*,b_*)$:
\begin{equation}\label{eq-BVP-ab}
\left\{\!\begin{aligned}
& \tri^2a_{\ell+1}=\f1{M_2+\ka_2}\pa_t z(a_\ell,b_\ell)(\,\cdot\,,0)+\f{M_2}{M_2+\ka_2}\tri^2a_\ell,\\
& \tri b_{\ell+1}=\f1{M_1+\ka_1}z(a_\ell,b_\ell)(\,\cdot\,,0)+\f{M_1}{M_1+\ka_1}\tri b_\ell
\end{aligned}\right.\quad\mbox{in }\Om,
\end{equation}
where $M_1,M_2>0$ are large parameters guaranteeing the convergence of \eqref{eq-BVP-ab}. By terminating \eqref{eq-BVP-ab} appropriately, one can obtain a good approximation of $(a_*,b_*)$. Parallel to the iterative update \eqref{eq-BVP-f} for recovering one profile, the implementation of \eqref{eq-BVP-ab} only involves the solutions of $2$ evolution equations, that is, \eqref{eq-IBVP-v2} and \eqref{eq-ibvp-z2}.

We close this subsection by summarizing the numerical reconstruction scheme for Problem \ref{prob-IMSP} with two unknown profiles as follows.

\begin{algo}\label{algo-2}
Choose a tolerance $\ep>0$, two regularization parameters $\ka_1,\ka_2>0$ and two tuning parameters $M_1,M_2>0$. Give an initial guess $(a_0,b_0)\in H_0^2(\Om)\times H_0^1(\Om)$ (e.g., $a_0=b_0\equiv0$), and set $\ell=0$.
\begin{enumerate}
\item Compute $(a_{\ell+1},b_{\ell+1})$ by the iterative update \eqref{eq-BVP-ab}.
\item If $\max\{\|a_{\ell+1}-a_\ell\|_{H^2(\Om)}/\|a_\ell\|_{H^2(\Om)},\|b_{\ell+1}-b_\ell\|_{H^1(\Om)}/\|b_\ell\|_{H^1(\Om)}\}<\ep$, stop the iteration. Otherwise, update $\ell\leftarrow\ell+1$ and return to Step 1.
\end{enumerate}
\end{algo}

\Subsection{Elliptic approach to solving the convection equation}\label{sec-convection}

Supposing that the initial values of \eqref{eq-IBVP-v2} are known, in this subsection we develop an efficient numerical scheme to reconstruct the two unknown profiles. By the relation \eqref{eq-ab-fg}, it suffices to solve the convection equation
\begin{equation}\label{eq-convection}
\bm r\cdot\nb f=c:=b+\bm p\cdot\nb a\in H_0^1(\Om),\quad f\in H_0^2(\Om)
\end{equation}
for $f$ and thus $g=a-f$, where $\BR^d\ni\bm r:=\bm q-\bm p\ne\bm0$. Since the unknown function $f$ is sufficiently smooth and vanishes on $\pa\Om$, we attempt to transform the first order equation \eqref{eq-convection} to some second order elliptic equations for better stability.

To this end, we introduce the change of variables
\begin{equation}\label{eq-coordinate}
\bm\xi=\bm\xi(\bm x):=\bm Q^\T\bm x,\quad\mathrm{SO}(d)\ni\bm Q:=\begin{pmatrix}
\displaystyle\f{\bm r}{|\bm r|} & *
\end{pmatrix},\quad\wh\Om:=\{\bm Q^\T\bm x\mid\bm x\in\Om\},
\end{equation}
where $*$ denotes some $d\times(d-1)$ matrix. In other words, $\bm Q$ is a $d\times d$ orthogonal matrix whose first column is normalized along $\bm r$. Correspondingly, defining the auxiliary functions
\[
\wh f(\bm\xi):=f(\bm x(\bm\xi))=f(\bm Q\bm\xi),\quad\wh c(\bm\xi):=c(\bm x(\bm\xi))=c(\bm Q\bm\xi),
\]
we know $\wh f\in H_0^2(\wh\Om)$ and $\wh c\in H_0^1(\wh\Om)$ due to the invertibility and the smoothness of this change of variables. Especially, by \eqref{eq-convection} we calculate
\begin{equation}\label{eq-xi1}
\f\pa{\pa\xi_1}\wh f(\bm\xi)=\f\pa{\pa\xi_1}\left(\f{\xi_1\bm r}{|\bm r|}\right)\cdot\nb_{\bm x}f(\bm Q\bm\xi)=\f{\bm r}{|\bm r|}\cdot\nb_{\bm x}f(\bm Q\bm\xi)=\f{c(\bm Q\bm\xi)}{|\bm r|}=\f{\wh c(\bm\xi)}{|\bm r|},\quad\bm\xi\in\wh\Om.
\end{equation}
Noting that \eqref{eq-xi1} only involves the first derivative in $\xi_1$, we write $\bm\xi=(\xi_1;\bm\xi')$, where $\bm\xi'\in\BR^{d-1}$ is regarded as a parameter. Therefore, we can understand the partial derivative $\f\pa{\pa\xi_1}$ in \eqref{eq-xi1} as an ordinary one, and differentiating both sides of \eqref{eq-xi1} with respect to $\xi_1$ leads to a series of boundary value problems for second order ordinary differential equations with parameters $\bm\xi'$:
\begin{equation}\label{eq-BVP-ODE}
\left\{\!\begin{alignedat}{2}
& \f{\rd^2\wh f(\xi_1;\bm\xi')}{\rd\xi_1^2}=\f1{|\bm r|}\f{\rd\wh c(\xi_1;\bm\xi')}{\rd\xi_1}, & \quad & (\xi_1;\bm\xi')\in\wh\Om,\\
& \wh f(\xi_1;\bm\xi')=0, & \quad & (\xi_1,\bm\xi')\in\pa\wh\Om.
\end{alignedat}\right.
\end{equation}
Hence, to reconstruct $\wh f$ in $\wh\Om$, it suffices solve \eqref{eq-BVP-ODE} line by line along $\xi_1$-axis for any fixed $\bm\xi'$, which is expected to be numerically cheap. As long as \eqref{eq-BVP-ODE} is solved for all $\bm\xi\in\wh\Om$, one can recover $f$ in $\Om$ by the relation $f(\bm x)=\wh f(\bm Q^\T\bm x)$ and finally $g=a-f$.

We summarize the above procedure by an illustrative example in Figure \ref{fig-geo}. First, we perform the coordinate transformation \eqref{eq-coordinate} so that the direction of the first coordinate $\xi_1$ coincides with that of $\bm r$. Next, direct calculations yield that the convection equation \eqref{eq-convection} in the original coordinate system is reduced to a special one \eqref{eq-xi1} in the new system, which is a series of first order ordinary differential equations with respect to $\xi_1$. Finally, utilizing the homogeneous boundary condition in \eqref{eq-convection}, we differentiate \eqref{eq-xi1} in $\xi_1$ and arrive at a series of boundary value problems \eqref{eq-BVP-ODE} of the second order. As is indicated by the shade in Figure \ref{fig-geo}, \eqref{eq-BVP-ODE} can be solved along each segment in $\Om$ on the direction of $\bm r$.
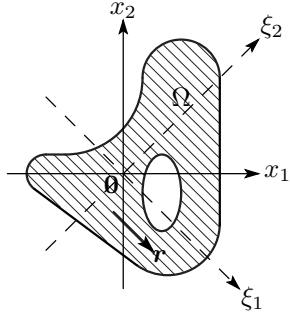
\begin{figure}[htbp]\centering
%WinTpicVersion4.32a
{\unitlength 0.1in%
\begin{picture}(13.0000,15.2300)(1.0000,-15.0000)%
% CIRCLE 1 0 3 0 Black White  
% 4 400 400 400 800 400 1700 1800 400
% 
\special{pn 13}%
\special{ar 400 400 400 400 6.2831853 1.5707963}%
% CIRCLE 1 0 3 0 Black White  
% 4 1000 400 800 400 1700 400 300 400
% 
\special{pn 13}%
\special{ar 1000 400 200 200 3.1415927 6.2831853}%
% CIRCLE 1 0 3 0 Black White  
% 4 300 900 400 900 300 400 0 1300
% 
\special{pn 13}%
\special{ar 300 900 100 100 2.2142974 4.7123890}%
% LINE 1 0 3 0 Black White  
% 2 300 800 400 800
% 
\special{pn 13}%
\special{pa 300 800}%
\special{pa 400 800}%
\special{fp}%
% CIRCLE 1 0 3 0 Black White  
% 4 930 1160 1200 1160 600 1610 1200 1160
% 
\special{pn 13}%
\special{ar 930 1160 270 270 6.2831853 2.2035452}%
% ELLIPSE 1 0 3 0 Black White  
% 4 900 1000 1000 800 1300 1000 1700 1000
% 
\special{pn 13}%
\special{ar 900 1000 100 200 0.0000000 6.2831853}%
% LINE 3 0 3 0 Black White  
% 68 1195 815 800 420 1195 755 805 365 1195 695 820 320 1195 635 845 285 1195 575 875 255 1195 515 910 230 1195 455 950 210 1195 395 1005 205 1180 320 1085 225 1195 875 795 475 1195 935 780 520 1195 995 765 565 1195 1055 745 605 885 805 720 640 845 825 695 675 825 865 665 705 810 910 630 730 800 960 595 755 800 1020 555 775 810 1090 510 790 1055 1395 455 795 1010 1410 400 800 965 1425 340 800 905 1425 285 805 820 1400 245 825 615 1255 215 855 375 1075 205 905 1090 1370 915 1195 1120 1340 955 1175 1150 1310 975 1135 1170 1270 990 1090 1185 1225 1000 1040 1195 1175 1000 980 1195 1115 990 910
% 
\special{pn 4}%
\special{pa 1195 815}%
\special{pa 800 420}%
\special{fp}%
\special{pa 1195 755}%
\special{pa 805 365}%
\special{fp}%
\special{pa 1195 695}%
\special{pa 820 320}%
\special{fp}%
\special{pa 1195 635}%
\special{pa 845 285}%
\special{fp}%
\special{pa 1195 575}%
\special{pa 875 255}%
\special{fp}%
\special{pa 1195 515}%
\special{pa 910 230}%
\special{fp}%
\special{pa 1195 455}%
\special{pa 950 210}%
\special{fp}%
\special{pa 1195 395}%
\special{pa 1005 205}%
\special{fp}%
\special{pa 1180 320}%
\special{pa 1085 225}%
\special{fp}%
\special{pa 1195 875}%
\special{pa 795 475}%
\special{fp}%
\special{pa 1195 935}%
\special{pa 780 520}%
\special{fp}%
\special{pa 1195 995}%
\special{pa 765 565}%
\special{fp}%
\special{pa 1195 1055}%
\special{pa 745 605}%
\special{fp}%
\special{pa 885 805}%
\special{pa 720 640}%
\special{fp}%
\special{pa 845 825}%
\special{pa 695 675}%
\special{fp}%
\special{pa 825 865}%
\special{pa 665 705}%
\special{fp}%
\special{pa 810 910}%
\special{pa 630 730}%
\special{fp}%
\special{pa 800 960}%
\special{pa 595 755}%
\special{fp}%
\special{pa 800 1020}%
\special{pa 555 775}%
\special{fp}%
\special{pa 810 1090}%
\special{pa 510 790}%
\special{fp}%
\special{pa 1055 1395}%
\special{pa 455 795}%
\special{fp}%
\special{pa 1010 1410}%
\special{pa 400 800}%
\special{fp}%
\special{pa 965 1425}%
\special{pa 340 800}%
\special{fp}%
\special{pa 905 1425}%
\special{pa 285 805}%
\special{fp}%
\special{pa 820 1400}%
\special{pa 245 825}%
\special{fp}%
\special{pa 615 1255}%
\special{pa 215 855}%
\special{fp}%
\special{pa 375 1075}%
\special{pa 205 905}%
\special{fp}%
\special{pa 1090 1370}%
\special{pa 915 1195}%
\special{fp}%
\special{pa 1120 1340}%
\special{pa 955 1175}%
\special{fp}%
\special{pa 1150 1310}%
\special{pa 975 1135}%
\special{fp}%
\special{pa 1170 1270}%
\special{pa 990 1090}%
\special{fp}%
\special{pa 1185 1225}%
\special{pa 1000 1040}%
\special{fp}%
\special{pa 1195 1175}%
\special{pa 1000 980}%
\special{fp}%
\special{pa 1195 1115}%
\special{pa 990 910}%
\special{fp}%
% VECTOR 1 0 3 0 Black White  
% 2 650 1100 850 1300
% 
\special{pn 13}%
\special{pa 650 1100}%
\special{pa 850 1300}%
\special{fp}%
\special{sh 1}%
\special{pa 850 1300}%
\special{pa 817 1239}%
\special{pa 812 1262}%
\special{pa 789 1267}%
\special{pa 850 1300}%
\special{fp}%
% STR 2 0 3 0 Black White  
% 4 1000 450 1000 500 5 0 0 0
% $\Omega$
\put(10.0000,-5.0000){\makebox(0,0){$\Omega$}}%
% STR 2 0 3 0 Black White  
% 4 850 1250 850 1300 1 0 0 0
% $\bm r$
\put(8.5000,-13.0000){\makebox(0,0)[lt]{$\bm r$}}%
% LINE 1 0 3 0 Black White  
% 2 1200 400 1200 1170
% 
\special{pn 13}%
\special{pa 1200 400}%
\special{pa 1200 1170}%
\special{fp}%
% LINE 1 0 3 0 Black White  
% 2 240 980 770 1378
% 
\special{pn 13}%
\special{pa 240 980}%
\special{pa 770 1378}%
\special{fp}%
% VECTOR 2 0 3 0 Black White  
% 2 100 900 1400 900
% 
\special{pn 8}%
\special{pa 100 900}%
\special{pa 1400 900}%
\special{fp}%
\special{sh 1}%
\special{pa 1400 900}%
\special{pa 1333 880}%
\special{pa 1347 900}%
\special{pa 1333 920}%
\special{pa 1400 900}%
\special{fp}%
% VECTOR 2 0 3 0 Black White  
% 2 700 1500 700 100
% 
\special{pn 8}%
\special{pa 700 1500}%
\special{pa 700 100}%
\special{fp}%
\special{sh 1}%
\special{pa 700 100}%
\special{pa 680 167}%
\special{pa 700 153}%
\special{pa 720 167}%
\special{pa 700 100}%
\special{fp}%
% VECTOR 2 1 3 0 Black White  
% 2 300 500 1300 1500
% 
\special{pn 8}%
\special{pa 300 500}%
\special{pa 1300 1500}%
\special{da 0.070}%
\special{sh 1}%
\special{pa 1300 1500}%
\special{pa 1267 1439}%
\special{pa 1262 1462}%
\special{pa 1239 1467}%
\special{pa 1300 1500}%
\special{fp}%
% VECTOR 2 1 3 0 Black White  
% 2 300 1300 1400 200
% 
\special{pn 8}%
\special{pa 300 1300}%
\special{pa 1400 200}%
\special{da 0.070}%
\special{sh 1}%
\special{pa 1400 200}%
\special{pa 1339 233}%
\special{pa 1362 238}%
\special{pa 1367 261}%
\special{pa 1400 200}%
\special{fp}%
% STR 2 0 3 0 Black White  
% 4 600 950 600 1000 2 0 0 0
% $\bm0$
\put(6.0000,-10.0000){\makebox(0,0)[lb]{$\bm0$}}%
% STR 2 0 3 0 Black White  
% 4 1500 850 1500 900 5 0 0 0
% $x_1$
\put(15.0000,-9.0000){\makebox(0,0){$x_1$}}%
% STR 2 0 3 0 Black White  
% 4 700 0 700 50 5 0 0 0
% $x_2$
\put(7.0000,-0.5000){\makebox(0,0){$x_2$}}%
% STR 2 0 3 0 Black White  
% 4 1370 1500 1370 1550 5 0 0 0
% $\xi_1$
\put(13.7000,-15.5000){\makebox(0,0){$\xi_1$}}%
% STR 2 0 3 0 Black White  
% 4 1470 80 1470 130 5 0 0 0
% $\xi_2$
\put(14.7000,-1.3000){\makebox(0,0){$\xi_2$}}%
\end{picture}}%
\caption{A two-dimensional illustration of the idea of transforming the convection equation \eqref{eq-convection} into a series of second order ordinary differential equations \eqref{eq-BVP-ODE}.}\label{fig-geo}
\end{figure}

\begin{ex}
(1) For $d=1$, it suffices to differentiate \eqref{eq-convection} to obtain a two-point boundary value problem
\[
\begin{cases}
r\,f''=c' & \mbox{in }\Om,\\
f=0 & \mbox{on }\pa\Om,
\end{cases}
\]l
which is automatically included in the framework of \eqref{eq-BVP-ODE}.

(2) For $d=2$, consider $\Om:=\{|\bm x|<1\}$ and $\bm r:=(1,1)^\T$. Then by the above argument, one can choose $\displaystyle\bm Q=\f1{\sqrt2}\begin{pmatrix}
1 & -1\\
1 & 1
\end{pmatrix}$, and \eqref{eq-BVP-ODE} becomes a series of two-point boundary value problems with the parameter $\xi_2\in(-1,1)$:
\[
\left\{\!\begin{alignedat}{2}
& \f{\rd^2\wh f(\xi_1;\xi_2)}{\rd\xi_1^2}=\f1{\sqrt2}\f{\rd\wh c(\xi_1;\xi_2)}{\rd\xi_1}, & \quad & |\xi_1|<\sqrt{1-\xi_2^2}\,,\\
& \wh f(\xi_1;\xi_2)=0, & \quad & |\xi_1|=\sqrt{1-\xi_2^2}\,.
\end{alignedat}\right.
\]
\end{ex}
%%%%%%%%%%%%%%%%%%%%%%%%%%%%%%%%%%%%%%%%%%%%%%%%%%

\Section{Concluding remarks}\label{sec-rem}

Regardless of the vast difference mainly in quantitative analysis, many inverse problems for (time-fractional) evolution equations share certain similarity in methodology and qualitative properties. Accordingly, for some specified inverse problems, it seems reasonable and possible to develop universal numerical reconstruction scheme valid for any fractional orders $\al\in(0,2]$. As a typical candidate, in \cite{LHY20} and this article, we investigated the inverse moving source problem on determining moving profile(s) in \eqref{eq-IBVP-u}--\eqref{eq-def-F}, which turns out to be novel for (time-fractional) evolution equations from both theoretical and numerical aspects. Guaranteed by the uniqueness claimed in \cite{LHY20}, here we adopted a conventional minimization procedure to construct iterative thresholding schemes (Algorithms \ref{algo-1} and \ref{algo-2}) for recovering one or two unknown profile(s), which follows the same line as that in \cite{JLLY17}. However, unlike \cite{JLLY17}, here we actually proposed two independent numerical methods, namely,
\begin{enumerate}
\item an iterative method for classical backward problems on determining initial values, and
\item an elliptic method for convection equations with whole boundary conditions.
\end{enumerate}
The first method seems novel for the simultaneous reconstruction of two initial values in the case of $1<\al<2$, and the second one is expected to improve the numerical stability and efficiency. Especially, since usual convection equations only require the in-flow condition on a partial boundary, to the author's knowledge there seems no literature on this direction.

Finally, we summarize several important future topics related to this article.
\begin{enumerate}
\item Needless to say, the immediate task should be the implementation of the proposed algorithms. Examining the whole process, we shall apply some mollification method for differentiating the noisy data, and prepare a forward solver for $1<\al<2$. Meanwhile, it would be interesting to observe the numerical performance as $\al\to2$. If the equation behaves closer to a wave equation in the sense of the finite wave propagation speed, it should require more observation time for stable recovery.
\item Similarly to \cite{LHY20}, the numerical schemes proposed in previous sections work for more general formulation than \eqref{eq-IBVP-u}, e.g., an elliptic part with constant coefficients instead of $-\tri$. However, the key introduction \cite[eq.\! (4.2)]{LHY20} of auxiliary functions fails whenever the governing operator involves variable coefficients. It would be interesting to find alternative ways to overcome this difficulty.
\item The assumption \eqref{eq-def-F} on the translation of moving sources seems restrictive and mostly unrealistic. Inspired by \cite{HKLZ19}, one can consider the replacement of \eqref{eq-def-F} e.g.\! with
    \begin{equation}\label{eq-def-F-new}
    F(\bm x,t)=\begin{cases}
    f(\bm x-\bm\rho(t))\,\te(t), & 0<\al\le1,\\
    f(\bm x-\bm\rho(t))\,\te(t)+g(\bm x-\bm\ze(t))\,\eta(t), & 1<\al\le2,
    \end{cases}
    \end{equation}
    where $\bm\rho,\bm\ze$ denote the orbits of moving sources, and $\te,\eta$ model the time evolution of source magnitude.
\item As another type of inverse moving source problems, the identification of moving orbits was studied in \cite{HLY20}, whereas its numerical reconstruction is absent except for hyperbolic equations. In the framework of \eqref{eq-def-F-new}, the orbit inverse problems deserve reconsideration especially from a numerical viewpoint.
\item In \cite{LHY20} and this article, we considered the partial interior data and required the observation subdomain to surround the whole boundary. Similarly, one can deal with the case of full lateral Cauchy data and investigate the same problem both theoretically and numerically.
\item On the direction of Subsection \ref{sec-convection}, it would be challenging to develop an elliptic approach to boundary value problems for convection equations with variable coefficients, i.e., $\bm r=\bm r(x)$ in \eqref{eq-convection}. It is conjectured that if the flows determined by $\bm r$ do not intersect each other and possess ergodicity, then we can reduce \eqref{eq-convection} into a series of second order ordinary differential equations on flows. In such a way, the convection equation can be related with elliptic equations with Laplace-Beltrami operators on Riemannian manifolds.
\end{enumerate}
\bigskip

{\bf Acknowledgement}\ \ The author appreciates the valuable discussions with Gen Nakamura (Hokkaido University), Masahiro Yamamoto (The University of Tokyo) and Guanghui Hu (Beijing Computational Science Research Center).
%%%%%%%%%%%%%%%%%%%%%%%%%%%%%%%%%%%%%%%%%%%%%%%%%%

% ---- Bibliography ----

\bigskip

{\small\noindent Research Institute for Electronic Science\\
Hokkaido University\\
N12W7, Kita-Ward, Sapporo 060-0812, JAPAN\\
E-mail address: ykliu@es.hokudai.ac.jp}

\end{document}